\documentclass[preprint, 12pt, english]{elsarticle}
\usepackage{amsmath}
\usepackage{latexsym, amssymb, mathtools}
\usepackage{amsthm}
\usepackage{color}
\usepackage{txfonts}
\usepackage{soul}

\newtheorem{thm}{Theorem}[section] 

\newtheorem{conj}[thm]{Conjecture}
\newtheorem{cor}[thm]{Corollary}

\newtheorem{exmpl}[thm]{Example}
\newtheorem{lem}[thm]{Lemma}
\newtheorem{prop}[thm]{Proposition}

\theoremstyle{definition}
\newtheorem{rem}[thm]{Remark}

\newcommand\operA[2]{{\if!#2!\operatorname{#1}\else{\operatorname{#1}_{#2}^{\phantom{I}}}\fi}} 

\newcommand\Cref[1]{{Corollary~\ref{#1}}}%

\def\tr{{\operatorname{Tr}}}

\newcommand{\Trace}[1][]{\if!#1!\operatorname{Tr}\else{\operatorname{Tr}_{#1}^{\phantom{I}}}\fi} 

\long\def\forget#1\forgotten{{}} %

\def\({\left(}
\def\){\right)}

\def\X{\mathcal{X}}
\usepackage{tabu}

\newif\iffurther
\furtherfalse

\newif\ifXY 
\XYtrue     
\ifXY

\input xy
\input xyidioms.tex
\usepackage{xy}
\xyoption{all} %
\fi 

\usepackage{longtable}
\usepackage{babel}

\journal{Communications in Algebra}

\begin{document}

\begin{frontmatter}

\title{Tensor Products of Cyclic Algebras of Degree 4 and their Kummer Subspaces}

\author{Adam Chapman}
\ead{adam1chapman@yahoo.com}
\author{Charlotte Ure}
\ead{charl.ure@gmail.com}

\address{Department of Mathematics, Michigan State University, East Lansing, MI 48824}

\begin{abstract}
We prove that the maximal dimension of a Kummer space in the generic tensor product of $n$ cyclic algebras of degree 4 is $4 n+1$.
\end{abstract}

\begin{keyword}
Division Algebras, Kummer Spaces, Generic Algebras, Graphs
\MSC[2010] primary 16K20; secondary 05C38,16W60
\end{keyword}

\end{frontmatter}

\section{Introduction}

Let $d$ be a positive integer and $F$ be an infinite field of characteristic 0 containing a primitive $d$th root of unity $\rho$. A cyclic or symbol algebra over $F$ of degree $d$ is an algebra that has a presentation $ F \left< x , y : x^d = \alpha, y^d = \beta, xy = \rho yx \right>$ for some $\alpha, \beta \in F^\times$. Denote this algebra by $\left( \alpha, \beta \right)_{d,F}$. \\

By the renowned Merkurjev-Suslin Theorem \cite{MS}, there is an isomorphism
\begin{align*}K_2(F) / d K_2(F)\ {\rightarrow}&\prescript{}{d}{Br}(F)\\
 \left\{ \alpha , \beta \right\} \mapsto& \left( \alpha, \beta \right)_{d,F}.\end{align*}
Consequently every central simple $F$-algebra of exponent $d$ is Brauer equivalent to a tensor product of cyclic algebras. A natural question is to determine the minimal number of cyclic algebras needed to express a given algebra $A$ of exponent $d$. This number is called the symbol length of $A$.\\

Consider a tensor product of cyclic algebras $A = \bigotimes_{k=1}^n \left( \alpha_k, \beta_k \right)_{d,F}$. An element $v \in A$ is called Kummer if $v^d \in F$. For example, the generators of the cyclic algebras in the product above are Kummer. An $F$-vector subspace of $A$ is called Kummer if it consists of Kummer elements. \\


Kummer spaces are connected to Clifford algebras (see for example \cite{ChapVish1}) and to the symbol length.
In \cite{Matzri} an upper bound for the symbol length of algebras of exponent $d$ over $C_r$-fields was provided based on the existence of certain Kummer spaces in tensor products of $n$ cyclic algebras of degree $d$.
The computation takes into account the maximal dimension of those Kummer spaces -- the larger the dimension, the sharper the bound.
For that purpose, Matzri considered the Kummer spaces described in Section \ref{Prem}, whose maximal dimension is $nd+1$.\\

Assume $A=\bigotimes_{k=1}^n \left( \alpha_k, \beta_k \right)_{d,F}$ is a division algebra. We conjecture the following:
\begin{conj}
The maximal dimension of a Kummer space in $A$ is $nd+1$.
\end{conj}
In the case $d=2$ the conjecture is known to be true as a result of the theory of Clifford algebras of quadratic forms. If $d=3$ and $n=1$, the conjecture holds true as well. This can be seen by considering the form $q(v)=\tr(v^2)$ over the space of trace zero elements. The form is nondegenerate and the space is 8-dimensional, so the Witt index is bounded by 4 from above. Since every Kummer space of dimension greater than 1 consists of trace zero elements and is totally isotropic with respect to $q(v)$, its dimension is bounded from above by the Witt index (see \cite{MV1} for more details).\\

The algebra $A$ is generic if $F=K(\alpha_k,\beta_k : 1 \leq k \leq n)$ is the function field in $2 n$ algebraically independent variables over some field $K$. In \cite{CGMRV} the conjecture was proved in the generic case for prime $d$ and $n=1$, and in \cite{KummerTensorChapman} it was proved in the generic case for $d=3$ and any $n$. \\

The goal of this paper is to prove the conjecture in the generic case for $d=4$ and any $n$. For each $k$ between $1$ and $n$, fix a pair of generators $x_k$ and $y_k$ for the algebra $(\alpha_k,\beta_k)_{d,F}$. A monomial in $A$ is an element of the form $\prod_{k=1}^n x_k^{a_k} y_k^{b_k}$ for some integers $a_k,b_k$ with $0 \leq a_k, b_k \leq d-1$. A Kummer subspace of $A$ is called monomial if it is spanned by monomials. In \cite[Theorem 4.1]{KummerTensorChapman} it was proved that the maximal dimension of a Kummer subspace of $A$ is equal to the maximal dimension of a monomial Kummer subspace of $A$. Hence we study the maximal dimension of monomial Kummer spaces.
 The algebraic question translates into a combinatorial question as can be seen in the following sections.

\section{Preliminaries}\label{Prem}

Let $d$ be an integer, $F$ be an infinite field of characteristic 0 and $A$ be a division tensor product of $n$ cyclic algebras of degree $d$ over $F$. Fix a presentation of $A$:
$$A = \bigotimes_{k=1}^n F\left< x_k, y_k : x_k^d = \alpha_k, y_k^d=\beta_k, x_ky_k = \rho y_k x_k \right>.$$

We use the following notation for the symmetric product, following \cite{Revoy}:
Given $v_1, \ldots, v_m \in A$, denote by $v_1^{d_1}* \cdots *v_m^{d_m}$ the sum of the products of $v_1, \ldots, v_m$, where each $v_k$ appears exactly $d_k$ times. We omit the superscript 1, e.g. $v_1*v_2^2 = v_1^1 * v_2^2= v_1 v_2^2 + v_1 v_2 v_1 + v_2^2 v_1$. 
A necessary and sufficient condition for $v_1,\dots,v_m$ to span a Kummer space is that $v_1^{d_1} * \dots * v_m^{d_m} \in F$ for any nonnegative integers $d_1,\dots,d_m$ with $d_1+\dots+d_m=d$.
If $v_1,\dots,v_m$ are monomials, then $v_1^{d_1} * \dots * v_m^{d_m}=c \cdot v_1^{d_1} \cdot \dots \cdot v_m^{d_m}$ where $c$ is the sum of $d$th roots of unity determined by the commutators of the $v_k$-s.
The condition $v_1^{d_1} * \dots * v_m^{d_m} \in F$ is then satisfied if and only if either $c=0$ or $v_1^{d_1} \cdot \dots \cdot v_m^{d_m} \in F$. The second option is satisfied if and only if the degrees of the different $x_k$-s and $y_k$-s add up to multiples of $d$. In this way the question becomes a combinatorial one.

\begin{exmpl}\label{mainexample}
Assume $v_1$ and $v_2$ span a Kummer space and $v_1 v_2=v_2 v_1$. Then $v_1^{d-1} * v_2=d v_1^{d-1} v_2$. Since the coefficient in this case is nonzero, $v_1^{d-1} v_2$ must be in $F$, which means that $F v_1=F v_2$, i.e. $v_1$ and $v_2$ are linearly dependent. 
\end{exmpl}

Monomial Kummer spaces of dimension $d n+1$ can be constructed in the following way: set $V_0=F$ and define $V_k$ recursively as $F[x_k] y_k + V_{k-1}x_k$ for any $k$ with $1 \leq k \leq n$. By induction one can show that each $V_k$ is Kummer: assume $V_{k-1}$ is Kummer. Every element in $V_K$ can be written as $f y_k+v x_k$ for some $f \in F[x_k]$ and $v \in V_{k-1}$. Since $(f y_k) (v x_k)=\rho (v x_k) (f y_k)$, we have $(f y_k+v x_k)^d=(f y_k)^d+(v x_k)^d$ by \cite[Remark 2.5]{ChapVish1}. Now $(f y_k)^d=N_{F[x_k]/F}(f) y_k^d \in F$ and $(v x_k)^d=v^d x_k^d \in F$. Note that the dimension of $V_n$ is $d n+1$.

\begin{rem}\label{mainremark}
\sloppy  By \cite[Theorem 4.1]{KummerTensorChapman}, every Kummer space in the generic tensor product of cyclic algebras gives rise to a monomial Kummer space of the same dimension.
This can be understood in terms of valuation theory (see \cite{TignolWadsworth:2015} for background):
Let $K$ be a field of characteristic 0 containing a primitive $d$th root of unity $\rho$.
Consider the field of iterated Laurent series $L=K((\alpha_1))((\beta_1))\dots((\alpha_n))((\beta_n))$ over $K$ and the algebra $$D=\bigotimes_{k=1}^n \prescript{}{L \enspace}{L} \langle x_k,y_k : x_k^d=\alpha_k, y_k^d=\beta_k, y_k x_k=\rho x_k y_k \rangle.$$ This algebra is the ring of twisted iterated Laurent series in variables $x_1,y_1,\dots,x_n,y_n$ subject to the relations $x_k x_\ell x_k^{-1} x_\ell^{-1}=y_k y_\ell y_k^{-1} y_\ell^{-1}=1$ and $y_k x_\ell y_k^{-1} x_\ell^{-1}=\rho^{\delta_{k,\ell}}$ for any $k,\ell \in \{1,\dots,n\}$ where $\delta_{k,\ell}$ is the Kronecker delta. In particular $D$ is a division algebra (see \cite[Section 1.1.3]{TignolWadsworth:2015}). The $(\alpha_1,\dots\beta_n)$-adic Henselian valuation on $L$ extends uniquely to $D$, and $D$ is totally ramified with respect to this valuation. Each $v \in D$ has a term of minimal value $\widetilde{v}$. Let $V$ be a Kummer space in $D$. Let $\widetilde{V}$ be the $L$-span of all $\widetilde{v}$ with $v \in V$. For any $v_1,\dots,v_m \in V$ and nonnegative integers $d_1,\dots,d_m$ with $d_1+\dots+d_m=d$, the element $\widetilde{v_1}^{d_1} * \dots * \widetilde{v_m}^{d_m}$ is either 0 or equal to $\widetilde{v_1^{d_1} * \dots * v_m^{d_m}}$. In both cases, $\widetilde{v_1}^{d_1} * \dots * \widetilde{v_m}^{d_m} \in L$, and so $\widetilde{V}$ is a Kummer space. Since each $\widetilde{v}$ is monomial, $\widetilde{V}$ is a monomial Kummer space. Since $D$ is totally ramified, by \cite[Proposition 3.14]{TignolWadsworth:2015}, for any $L$-subspace $W$ of $D$ we have $|\Gamma_W/\Gamma_L|=[W:L]$ where $\Gamma_W$ is the set of values of all $v \in W \setminus \{0\}$. Since $\Gamma_{V}=\Gamma_{\widetilde{V}}$ we have $[V:L]=[\widetilde{V}:L]$.
Now let $F=K(\alpha_1,\beta_1,\dots,\alpha_n,\beta_n)$ be the function field in $2n$ algebraically independent variables over $K$ and let $$A=\bigotimes_{k=1}^n \prescript{}{F\enspace}{F\langle x_k,y_k : x_k^d=\alpha_k, y_k^d=\beta_k, y_k x_k=\rho x_k y_k \rangle}.$$ For every Kummer space $V$ in $A$, $V \otimes_F L$ is a Kummer space in $D$ with $[V \otimes_F L:L]=[V:F]$. Therefore $\widetilde{V \otimes_F L}$ is a monomial Kummer space in $D$ with $[V \otimes_F L:L]=[\widetilde{V \otimes_F L}:L]$. Finally, $\widetilde{V \otimes_F L} \cap A$ is a monomial Kummer space in $A$ and $[\widetilde{V \otimes_F L}:L]=[\widetilde{V \otimes_F L} \cap A:F]$.
\end{rem}

\section{Graphs of Monomial Kummer Spaces}

Let $n$ be a positive integer, $F$ 
a field of characteristic 0 containing a primitive fourth root of unity $i$, and $A$  
 a division tensor product of $n$ cyclic algebras of degree $4$ over $F$. Fix a presentation of $A$ and let $\X$ be the set of monomials $\left\{\prod_{k=1}^n x_k^{a_k} y_k^{b_k} : 0 \leq a_k,b_k \leq 3 \right\}$ with respect to that presentation. \\

Following \cite{KummerTensorChapman}, we associate a directed labeled graph $(\X, E)$ with the Kummer elements of $A$, where the vertices are elements of $\X$ and the edges are of two types:
\begin{align*}
&\xymatrix{
v \ar@{->}[r] &w &\text{if } vw = i wv,}\text{and }\\
&\xymatrix{v \ar@{--}[r]& w &\text{if } vw = - wv.\\
}
\end{align*}
 
For any $B \subset \X$ denote by $\left(B, E_B\right)$ the subgraph obtained by taking all vertices in $B$ and all edges between them.

From now on suppose that $B \subset \X$ is the basis of a monomial Kummer subspace of $A$. Clearly for every two distinct monomials $v,w \in B$ we have $vw = i^k wv$ for some $k$ in $\left\{0,1,2,3\right\}$.
We want to understand the graph $\left(B, E_B\right)$, so that we can bound the number of vertices in this graph, and consequently the dimension of the monomial Kummer space.

\section{Basic Properties}

\begin{rem} For any two distinct $v,w \in B$, one of the following holds:
\begin{enumerate}
\item $\xymatrix{v \ar[r] & w}$, or
\item $ \xymatrix{v & w\ar[l]}$, or 
\item $\xymatrix{v \ar@{--}[r]&w}$ and $v^2 w^2 \in F$. 
\end{enumerate}\end{rem}

\begin{proof}
By Example \ref{mainexample}, $v$ and $w$ do not commute.
If $\xymatrix{v \ar@{--}[r]&w}$, we get $v^3*w=0, v^2*w^2 =2v^2w^2$ and $v*w^3=0$. Hence $v^2w^2 \in F$.
\end{proof}

\begin{lem}\label{anti-commute}
Any $z \in B$ can anti-commute with at most one other element in $B$.
\end{lem}

\begin{proof}
Assume the contrary, that it anti-commutes with two different elements $v,w \in B$.
There are two cases to check, $\xymatrix{v \ar@{--}[r]&w}$ and $\xymatrix{v \ar@{->}[r]&w}$.

If $\xymatrix{v \ar@{->}[r]&w}$ then from the relation $v * w *z^2 \in F$ we obtain $(2-2 i) v w z^2 \in F$ and so $v w z^2 \in F$. Since $z^2 w^2 \in F$, we obtain $v^{-1} w \in F$, which means $w \in F v$, which is a contradiction.

Hence we may assume that $\xymatrix{v \ar@{--}[r]&w}$, and so $F v^2=F w^2=F z^2$. Therefore $v^{-1} w$ and $v^{-1} z$ are scalar multiples of monomials with even powers of $x_k$ and $y_k$ for every $k$.
Therefore they commute.
From the equality $v^{-1} w v^{-1} z=v^{-1} z v^{-1} w$ we obtain $v^{-2} w z=v^{-2} z w$, which means $w z=z w$, contradiction.
\end{proof}

\begin{prop}\label{threeelements}
For every three distinct elements $v,w,z \in B$, the following case is impossible: $$
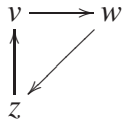
\xymatrix{v \ar@{->}[r] &w\ar@{->}[ld] \\ z \ar@{->}[u] & }$$
\end{prop}

\begin{proof}
In this case, from the relations $v^2 * w * z,v * w^2 * z, v * w * z^2 \in F$ we get $(4-4i) v^2 w z,(-4-4i) v w^2 z,(4-4i) v w z^2 \in F$, and so $v^2 w z,v w^2 z,v w z^2 \in F$.
Therefore $(v^2 w z)^{-1} (v w^2 z) \in F$, which means that $v^{-1} w \in F$. Hence $w \in F v$, contradiction.
\end{proof}

\begin{cor}\label{dircy}
If $\left\{v_1,v_2,v_3,v_4\right\}$ is a directed simple cycle then $\xymatrix{v_1\ar@{--}[r] & v_3}$ and $\xymatrix{v_2\ar@{--}[r] & v_4}$.
\end{cor}

\begin{prop}\label{fourelements}
For every four distinct elements $v,w,z,t \in B$, the following two cases are impossible:\\
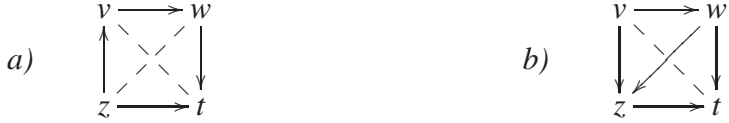
\begin{tabu}{X[m]X[8,l,m]X[m]X[8,l,m]}
a)& $\xymatrix{v \ar@{->}[r]\ar@{<-}[d]\ar@{--}[rd] &w \ar@{->}[d]\ar@{--}[ld] \\ z\ar@{->}[r] &t }$ & b)&$\xymatrix{v \ar@{->}[r]\ar@{->}[d]\ar@{--}[rd] &w \ar@{->}[d]\ar@{->}[ld] \\ z \ar@{->}[r]&t }$\end{tabu}
\end{prop}

\begin{proof}
\begin{enumerate}
\item[$a)$] From the relation $v * w *z * t \in F$ we obtain $-4 i v w z t \in F$, which means that $vwzt \in F$. Since also $v^2 w z \in F$, we have $t v^{-1} \in F$, and so $t \in Fv$, contradiction. 
\item[$b)$] From the relation $v*w*z*t \in F$ we obtain $(-4-4i) v w z t \in F$, which means $vwzt \in F$. Since also $w^2 v t \in F$, we have $w z^{-1} \in F$, and so $w \in F z$, contradiction.
\end{enumerate}
\end{proof}

\section{Directed Cycles}

\begin{lem}\label{CycleAllArrows} Suppose for some $v_1, v_2, v_3, v_4 \in B$ we have
$$
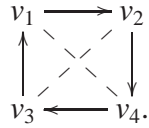
\xymatrix{v_1 \ar@{->}[r]\ar@{<-}[d]\ar@{--}[rd] &v_2 \ar@{->}[d]\ar@{--}[ld] \\v_3 \ar@{<-}[r]&v_4.}$$
Then for every other $w \in B$ either 
\begin{enumerate}
\item $\xymatrix{w \ar[r]& v_k}$ for $k=1,2,3,4$ or
\item $\xymatrix{w & v_k\ar[l]}$ for $k=1,2,3,4$.
\end{enumerate}\end{lem}

\begin{proof}
By Lemma \ref{anti-commute}, $w v_k = -v_k w$  is not possible for any $k$. 
Suppose $\xymatrix{w \ar[r]& v_1}$ and $\xymatrix{v_2 \ar[r]& w}$. Then we have $\xymatrix{w \ar[d] & v_2\ar[l]\\ v_1 \ar[ru]}$ contradictory to Proposition \ref{threeelements}.
\end{proof}

\begin{prop}\label{noOtherCycle}
There are no directed cycles of length other than four in $\left(\X, E_B\right)$. \end{prop}

\begin{proof} 
We already know that directed cycles of length less than four do not exist.
Suppose there is
$$
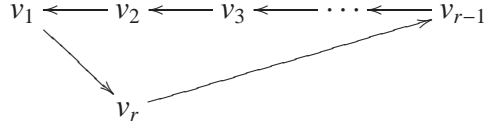
\xymatrix{v_1 \ar@{<-}[r] &v_2 \ar@{<-}[r] &v_3 \ar@{<-}[r] & \cdots \ar@{<-}[r] & v_{r-1} \ar@{<-}[llld]\\ & v_r \ar@{<-}[lu]}$$
for $r \geq 5$. Let $k$ be the maximal integer with $2 \leq k \leq r-1$ such that $\xymatrix{v_k \ar[r]& v_1}$. Now if $\xymatrix{v_1 \ar[r]& v_{k+1}}$ we get a three-cycle
$$
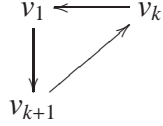
\xymatrix{v_1 \ar@{<-}[r]& v_k \ar@{<-}[ld]\\ v_{k+1} \ar@{<-}[u] }$$ which is impossible. Hence $v_{k+1}v_1 = -v_1 v_{k+1}$ and in particular $k \leq r-2$. By the maximality of $k$, $\xymatrix{v_{k+2} \ar[r]& v_1}$ is impossible. Since $v_1 v_{k+1} = - v_{k+1}v_1$, we have $v_{k+2} v_1 \neq -v_1 v_{k+2}$ and we get the cycle
$$
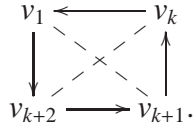
\xymatrix{v_1 \ar@{<-}[r] \ar[d] \ar@{--}[rd] & v_k \ar@{<-}[d] \ar@{--}[ld]\\v_{k+2} \ar[r] &v_{k+1}\vspace{1em} .}$$
If $k=2$ then $\xymatrix{v_5 \ar[r]& v_4}$ and so also by Lemma \ref{CycleAllArrows} $\xymatrix{v_5 \ar[r]& v_1}$, contradictory to the maximality of $k$. Hence $k \geq 3$.

Let $j$ be the minimal integer with $3 \leq j \leq r$ such that $\xymatrix{v_1 \ar[r]& v_{j}}$. As before we get the four-cycle
$$
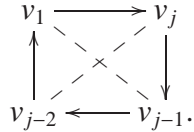
\xymatrix{v_1 \ar@{->}[r] \ar@{<-}[d] \ar@{--}[rd] & v_j \ar@{->}[d] \ar@{--}[ld]\\ v_{j-2} \ar@{<-}[r] &v_{j-1} .}$$
By maximality of $k$ we get $k\geq j-2$. If $k > j-2$ then $v_{k+1}$ and $v_{j-1}$ both anti-commute with $x_1$, contradictory to Lemma \ref{anti-commute}.

Assume $k=j-2$. Then $\xymatrix{v_{k-1} \ar[r]& v_1}$. However $\xymatrix{v_{k} \ar[r]& v_{k-1}}$ contradictory to Lemma \ref{CycleAllArrows}.
\end{proof}

\begin{lem}\label{blockII}
 Suppose for some $v_1, v_2, v_3, v_4 \in B$ we have
$$
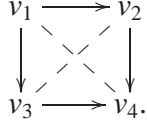
\xymatrix{v_1 \ar[r]\ar@{--}[rd]\ar@{->}[d]& v_2 \ar[d] \ar@{--}[ld] \\ v_3 \ar[r]& v_4 .}$$
Then for every other $w \in B$ either 
\begin{enumerate}
\item $\xymatrix{w \ar[r]& v_k}$ for $k=1,2,3,4$ or
\item $\xymatrix{w & v_k\ar[l]}$ for $k=1,2,3,4$.
\end{enumerate}
\end{lem}

\begin{proof}
By Lemma \ref{anti-commute} $wv_k = -v_kw$  is not possible for any $k$. \\
If $\xymatrix{w \ar[r]& v_1}$ then by Proposition \ref{threeelements} and Lemma \ref{anti-commute} $\xymatrix{w \ar[r]& v_2},$ $\xymatrix{w \ar[r]& v_3}$ and $\xymatrix{w \ar[r]& v_4}$.\\
From now on suppose $\xymatrix{v_1 \ar[r]& w}$. If $\xymatrix{w \ar[r]& v_2}$, then $\xymatrix{w \ar[r]& v_3}$ as well. But then we get
$$
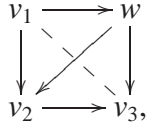
\xymatrix{ v_1 \ar[r] \ar@{--}[rd]\ar[d] & w \ar[ld]\ar[d] \\ v_2 \ar[r] & v_3,}$$ contradictory to Proposition \ref{fourelements}. Hence $\xymatrix{v_2 \ar[r]& w}$. Similarly, if $\xymatrix{w \ar[r]& v_3}$ we get $$
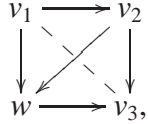
\xymatrix{ v_1 \ar[r] \ar@{--}[rd]\ar[d] & v_2 \ar[ld]\ar[d] \\ w \ar[r] & v_3,}$$ contradiction. Thus $\xymatrix{v_3 \ar[r]& w}$. By Proposition \ref{threeelements} $\xymatrix{v_4 \ar[r]& w}$.
\end{proof}

\section{Maximal dimension of a Monomial Kummer space}

The goal of this section is to bound the dimension of a monomial Kummer space.

\begin{lem}\label{path}Let $B' \subset B$ be a subset of cardinality $r$ satisfying the following condition: For all $v, w \in B$ either $\xymatrix{v \ar[r]& w}$ or $\xymatrix{w \ar[r]& v}$. Then the elements of $B'$ can be ordered $\left\{v_1,\dots,v_r\right\}$ such that $\xymatrix{v_k \ar[r]& v_j}$ for $k < j$.
\end{lem}

\begin{proof} Otherwise $B'$ contains a directed cycle, which is of length $4$ by \ref{noOtherCycle} and so by Corollary \ref{dircy} $B'$ contains pairs of anti-commuting elements, contradiction.
\end{proof}

\begin{lem}\label{quaternion}
Assume $B$ contains $\left\{v_1,v_2,\dots,v_{2n+1}\right\}$ such that $\xymatrix{ v_k \ar[r]& v_j}$ for $k<j$.  Suppose there exists another element $w \in B$ and an odd integer $r$ between $1$ and $2 n+1$ such that
\begin{enumerate}
\item  $\xymatrix{ v_k \ar[r]& w}$ for $k <r$, and
\item $\xymatrix{v_r \ar@{--}[r]& w,}$ and 
\item $\xymatrix{w \ar[r]& v_k}$ for $k >r$.
\end{enumerate}
Then the set $F\langle B\rangle$ contains a tensor product of $n$ cyclic algebras of degree 4 and one quaternion algebra over $F$.
\end{lem}

\begin{proof}
Let $\ell=\frac{r-1}{2}$. 
Now write 
$$\mu_j=\begin{cases} \left( \prod_{k=1}^{j-1} v_{2k-1} v_{2k}^{-1} \right) v_{2j-1} & j \leq \ell\\
\left( \prod_{k=1}^{\ell} v_{2k-1} v_{2k}^{-1} \right) \left( \prod_{k=\ell+1}^{j-1} v_{2k} v_{2k+1}^{-1} \right) v_{2 j} & \ell+1 \leq j \leq n \\
\left( \prod_{k=1}^{\ell} v_{2k-1} v_{2k}^{-1} \right) \left( \prod_{k=\ell+1}^{n} v_{2k}^{-1} v_{2k+1} \right) v_r & j=n+1, \enspace \text{and}  \end{cases}$$
$$\eta_j=\begin{cases} \left( \prod_{k=1}^{j-1} v_{2k-1} v_{2k}^{-1} \right) v_{2j} & j \leq \ell\\
\left( \prod_{k=1}^{\ell} v_{2k-1} v_{2k}^{-1} \right) \left( \prod_{k=\ell+1}^{j-1} v_{2k} v_{2k+1}^{-1} \right) v_{2 j+1} & \ell+1 \leq j \leq n \\
\left( \prod_{k=1}^{\ell} v_{2k-1} v_{2k}^{-1} \right) \left( \prod_{k=\ell+1}^{n} v_{2k}^{-1} v_{2k+1} \right) w & j=n+1.  \end{cases}$$
Clearly $F \langle v_1,\dots,v_{2n+1},w\rangle=F \langle \mu_1,\eta_1,\dots,\mu_{n+1},\eta_{n+1} \rangle$. Note that for each $j \in \{1,\dots,n+1\}$, $\eta_j$ and $\mu_j$ commute with all the other elements in $\left\{\mu_1,\eta_1,\dots,\mu_{n+1},\eta_{n+1}\right\}$ except each other. Therefore $F \langle \mu_1,\eta_1,\dots,\mu_{n+1},\eta_{n+1}\rangle$ decomposes as
$$F \langle \mu_1,\eta_1 \rangle \otimes \dots \otimes F \langle \mu_{n+1},\eta_{n+1}\rangle.$$
The first $n$ algebras in this decomposition are cyclic of degree $4$ and the last algebra is a quaternion algebra over $F$. 
\end{proof}

\begin{thm}\label{main}
The dimension of a monomial Kummer space in a division tensor product $A$ of $n$ cyclic algebras of degree 4 is at most $4 n+1$.
\end{thm}

\begin{proof}
Suppose there is a monomial Kummer space with basis $B$ of dimension $4n+2$. By Lemma \ref{anti-commute}, every element in $B$ can anti-commute with at most one other element. Note that the maximal number of such pairs in $B$ is $\left\lfloor\frac{ \left| B \right|}2\right\rfloor = 2n+1$. Let $B'$ be a subset obtained by taking off arbitrarily one element from each pair of anti-commuting elements in $B$. For every two elements $v$ and $y$ in $B'$ we have either $\xymatrix{v \ar[r]& y} $ or $\xymatrix{ v & y \ar[l]}$. Therefore, by Lemma \ref{path}, $B=\left\{v_1,v_2,\dots,v_{r}\right\}$ where $\xymatrix{v_k \ar[r]& v_j}$ for every $1 \leq k < j \leq r$ such that $r \geq 2 n+1$. \\
Suppose that $r > 2n+1$. Then $F\langle B\rangle$ contains the subalgebra $F \langle v_1,\dots,v_{2n+2} \rangle$ which decomposes as a tensor product of $n+1$ cyclic algebras of degree $4$
$$F\left<v_1, v_2\right> \otimes F \left<v_1 v_2^{-1} v_3 , v_1 v_2^{-1} v_4\right> \otimes \cdots \otimes F\left<\left( \prod_{k=1}^n v_{2k-1} v_{2k}^{-1}\right) v_{2n+1} , \left( \prod_{k=1}^n v_{2k-1}v_{2k}^{-1} \right) v_{2n+2} \right>,$$ contradictory to the degree of $A$. The decomposition can be explained in a similar way to Lemma \ref{quaternion}. \\

Therefore $r=2n+1$. We want to write down the graph of $B$, so we start with the path containing the elements of $B'$, and each $v_k$ has an anti-commuting partner $v_k'$:
$$\xymatrix{v_1 \ar[d] \ar@{--}[r] & v_1' \\ v_2 \ar[d] \ar@{--}[r]&v_2' \\ \vdots \ar[d] \\ v_{2n+1} \ar@{--}[r]& v_{2n+1}'}$$
We want to determine the remaining edges in this graph. 
Recall that every element can only anti-commute with at most one other element. By Proposition \ref{fourelements}, there is no subgraph of the form
$$\xymatrix{ v\ar[r] \ar[d] \ar@{--}[rd] & w  \ar@{--}[ld]   \\   s\ar[r] &t.  \ar[u] } $$

Let $B_k$ be the subgraph (``block") containing the vertices $v_k,v_k',v_{k+1},v_{k+1}'$.
A priori there are 8 options for each block $B_k$.
Four out of these options are of the forbidden type mentioned above.
Therefore each block $B_k$ can only be one of the following types:\\
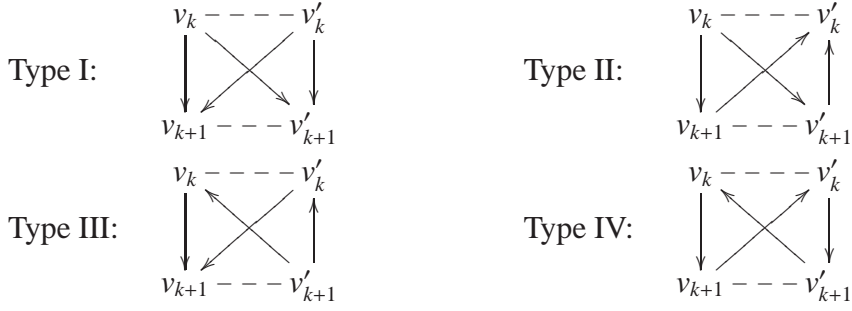
\begin{longtabu}{X[m]X[3,l,m]X[m]X[3,l,m]}
Type I: &$\xymatrix{v_k\ar[d]\ar@{--}[r]\ar[rd] & v_k'\ar[ld]\ar[d]\\v_{k+1}\ar@{--}[r] & v_{k+1}'}$ &Type II: &$\xymatrix{v_k\ar[d]\ar@{--}[r]\ar[rd] &v_k'\ar@{<-}[ld]\\v_{k+1}\ar@{--}[r] &v_{k+1}'\ar[u]}$\\
Type III: &$\xymatrix{v_k\ar[d]\ar@{--}[r]\ar@{<-}[rd] &v_k'\ar[ld]\\v_{k+1}\ar@{--}[r] &v_{k+1}'\ar[u]}$ &
Type IV: &$\xymatrix{v_k\ar[d]\ar@{--}[r]\ar@{<-}[rd] &v_k'\ar@{<-}[ld]\ar[d]\\v_{k+1}\ar@{--}[r] &v_{k+1}'}$
\end{longtabu}

Now we show that every possible combination leads to a contradiction.\\

\noindent\textbf{Step 1:} \textit{$B_1$ is of type IV or II:}\\
\noindent In the remaining cases, we have $\xymatrix{ v_1' \ar[r]& v_2}$ and hence $\xymatrix{v_1' \ar[r]& v_k}$ for all $k>1$. So by Lemma \ref{quaternion} the set $\left\{v_1',v_1, v_2, \ldots, v_{2n+1} \right\}$ generates over $F$ a tensor product of $n$ cyclic algebras of degree 4 and one quaternion algebra, contradiction.\\

\noindent\textbf{Step 2:} \textit{If $B_k$ is of type II or IV then $B_{k+1}$ is of type I:}\\
\noindent Assume $B_k$ is of type II.
If $B_{k+1}$ is of type II or IV, the element $v_{k+2}'$ contradicts Lemma \ref{blockII}, since either $\xymatrix{v_{k+1} \ar[r]& v_{k+2}'}$ and $\xymatrix{v_{k+2}' \ar[r]& v_{k+1}'}$ or vice versa.\\
If $B_{k+1}$ is of type III then we have
$$
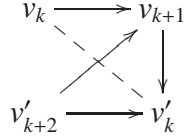
\xymatrix{v_k  \ar[r] & v_{k+1} \ar[d] \\ v_{k+2}' \ar[r] \ar[ru] & v_k' \ar@{--}[lu] }$$
and thus by Proposition \ref{fourelements} $\xymatrix{ v_{k+2'} \ar[r]& v_k.}$ This is a contradiction since by Proposition \ref{fourelements} the following is impossible
$$
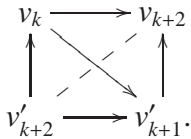
\xymatrix{v_k \ar[r] \ar[rd] & v_{k+2} \ar@{--}[ld] \\ v_{k+2}' \ar[u] \ar[r]& v_{k+1}'. \ar[u]}$$

Assume $B_k$ is of type IV.
If $B_{k+1}$ is of type IV, there is a second cycle. This is impossible since by Lemma \ref{anti-commute} all cycles are vertex-disjoint. \\
If $B_{k+1}$ is of type II then $v_k,v_{k+1},v_{k+2}',v_{k+1}'$ is a cycle which intersects the cycle $B_k$, contradiction.
Suppose $B_{k+1}$ is of type III. Then we have $$
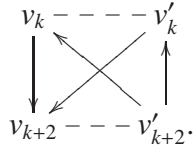
\xymatrix{v_k \ar[d] \ar@{--}[r]& v_k' \ar[ld] \\ v_{k+2} \ar@{--}[r] & v_{k+2}'. \ar[lu] \ar[u]}$$
This contradicts Lemma \ref{blockII} since $\xymatrix{v_{k+1} \ar[r]& v_{k+2}}$ and $\xymatrix{v_{k+2}' \ar[r]& v_{k+1}}$.\\

\noindent\textbf{Step 3: } \textit{If $B_k$ is of type I and $k$ is even, then $B_{k+1}$ is of type II or IV: }\\
Similar to Step 1 above, we get that if $B_{k+1}$ is of type I or III, we have a graph
$$
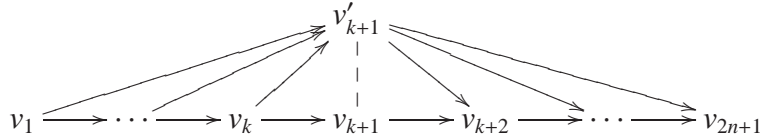
\xymatrix{
               &                    &                 & v_{k+1}' \ar@{--}[d] \ar[rd]\ar[rrd]\ar[rrrd]\\
v_1 \ar[rrru]\ar[r]& \cdots \ar[rru]\ar[r]& \ar[ru]v_k \ar[r]& v_{k+1} \ar[r]& v_{k+2} \ar[r] & \cdots \ar[r] & v_{2n+1} }
$$
and so the condition of Lemma \ref{quaternion} is satisfied,  contradiction.\\

\textbf{Step 4:} \textit{Finally, $B_{2n}$ has to be of type III or IV:}\\
If the last block is of type I or II, then the set $\left\{v_1, \ldots, v_{2n+1}, v_{2n+1}'\right\}$ satisfies the condition of Lemma \ref{quaternion} with $r=2n+1$, contradiction.\\

By Step 1 $B_1$ is of type II or IV. But then by Step 2 $B_2$ is of type I. By Step 3 $B_3$ is of type II or IV. This pattern continues for the following blocks. Now the number of blocks is even and so the last block, $B_{2n}$ is of type I, contradictory to Step 5.
\end{proof}

\begin{thm}
Let $F=K(\alpha_k,\beta_k : 1 \leq k \leq n)$ be the function field in $2 n$ variables over a field $K$ of characteristic 0 containing primitive 4th root of unity $i$, and let $A=\otimes_{k=1}^n \left(\alpha_k,\beta_k\right)_{4,k}$  be the generic tensor product of $n$ cyclic algebras of degree 4. The maximal dimension of a Kummer space in $A$ is $4 n+1$.
\end{thm}

\begin{proof}
By Remark \ref{mainremark} the dimension of a monomial Kummer space in $A$ is bounded by $4 n+1$. By \cite[Theorem 4.1]{KummerTensorChapman} it holds for any Kummer space in $A$.
\end{proof}

\section*{Acknowledgements}
\noindent The authors thank Adrian Wadsworth for his comments on the manuscript and for pointing out the connection to valuation theory. The authors also thank Rajesh Kulkarni and Jean-Pierre Tignol for commenting on the manuscript.

\section*{Bibliography}
\bibliographystyle{amsalpha}
\bibliography{bibfile}
\end{document}